\documentclass[a4paper]{article}

\usepackage[english]{babel}
\usepackage[utf8x]{inputenc}
\usepackage[T1]{fontenc}

\usepackage{amsmath, amsfonts, amssymb, amsthm}
\usepackage{graphicx}
\usepackage[colorinlistoftodos]{todonotes}
\usepackage[colorlinks=true, allcolors=blue]{hyperref}

\newtheorem{thm}{Theorem}
\newtheorem{lem}{Lemma}

\newtheorem{que}{Question}

\DeclareMathOperator{\conv}{conv}

\begin{document}

\title{New irrational polygons \\ with Ehrhart-theoretic period collapse}
\author{Quang-Nhat LE
\thanks{This project has received funding from the European Research Council (ERC) under the European Union's Horizon 2020 research and innovation programme (grant agreement No [ERC StG 716424 - CASe]).}
}

\maketitle

\begin{abstract}
In a recent paper, Cristofaro-Gardiner--Li--Stanley \cite{cristofaro2015new} constructed examples of irrational triangles whose Ehrhart function (i.e. lattice-point count) is a polynomial, when restricted to positive integer dilation factors. This is very surprising because the Ehrhart functions of rational polygons are usually only quasi-polynomials. We demonstrate that most of their triangles can also be obtained by a simple cut-and-paste procedure that allows us to build new examples with more sides. Our examples might potentially have applications in the theory of symplectic embeddings.
\end{abstract}

\section{Introduction}
The \textbf{Ehrhart function} of a polytope $P \subset \mathbb{R}^d$ is $L_P(t) := \#(tP \cap \mathbb{Z}^d)$, where $t > 0$ is the dilation factor. For the rest of this paper, the dilation factor $t$ will be a positive integer and our polytopes will be closed and possibly non-convex. 

The following theorem was proved by Eug\`ene Ehrhart and Ian Macdonald.
\begin{thm}[\cite{ehrhart1967probleme} and \cite{macdonald1971polynomials}] If $P$ is an \textbf{integer} polytope, i.e. $\text{vertices}(P) \subset \mathbb{Z}^d$, and $t$ is a positive integer, then $L_P(t)$ is a polynomial.

If $P$ is an \textbf{rational} polytope, i.e. $\text{vertices}(P) \subset \mathbb{Q}^d$, and $t$ is a positive integer, then $L_P(t) = a_d(t)t^d + \dots + a_0(t)$ is a \textbf{quasi-polynomial}, which means it is a ``polynomial'' whose coefficients $a_i(t), i = 0, \cdots, d$ are periodic functions on $t \in \mathbb{Z}_{>0}$. 
\end{thm}


We say that $P$ exhibits \textbf{Ehrhart-theoretic period collapse} if $P$ is not an integer polytope and its Ehrhart function $L_P(t)$ is a polynomial in $t$. This phenomenon has been studied extensively for rational polytopes by McAllister--Woods \cite{mcallister2003minimum}, Haase--McAllister \cite{mcallister2003minimum}, Beck--Sam--Woods \cite{beck2007maximal}, etc. 

For irrational polytopes, the behavior of the Ehrhart function can be very wild. Usually, one only expects an asymptotic, but not exact, formula for $L_P(t)$. Therefore, the work of Cristofaro-Gardiner--Li--Stanley \cite{cristofaro2015new} came as a startling surprise: they demonstrated a family of irrational triangles whose Ehrhart functions are simply polynomials, and thus, established the existence of period collapse in irrational polytopes. Note that the restriction $t \in \mathbb{Z}_{>0}$ is crucial here: if $t$ is allowed to vary in $\mathbb{R}_{>0}$, such a result is impossible.

\begin{thm}[\cite{cristofaro2015new}, part of Theorem 1.1]
	Consider the triangle $T = T_{\alpha,\beta}$ with vertices at $(0,0)$, $(h,0)$ and $(k,0)$ such that $1/h + 1/k = \alpha \in \mathbb{Z}_{>0}$, $h + k = \beta \in \mathbb{Z}_{>0}$, $h,k > 0$ and the slope $h/k$ is irrational.
    \begin{itemize}
    \item $L_T(t)$ is a quasi-polynomial.
    \item If $\alpha = 1$ or $(\alpha,\beta) \in \{(3,3),(2,4)\}$, then $L_T(t) = \frac{1}{2} \frac{\beta}{\alpha} t^2 + \frac{1}{2} \beta t + 1$ is a polynomial.
    \end{itemize}
\end{thm}

We call the triangles $T_{1,\beta}$ \textbf{common CGLS triangles} and $T_{3,3}, T_{2,4}$ \textbf{exceptional CGLS triangles}. We note that Theorem of 1.1 of \cite{cristofaro2015new} also contains a converse to the above statements.

In this paper, we will construct many more examples that include common CGLS triangles as special cases. 
\begin{thm}\label{thm:quad}
	Let $h,k > 0$ be two irrational numbers which sum up to an integer. Then, the quadrilateral $Q = Q_{h,k} = \conv\{(0,0),(h,0),(1,1),(0,k)\}$ exhibits Ehrhart-theoretic period collapse. More concretely, $L_Q(t) = \frac{h+k}{2} t^2 + \frac{h+k}{2} t + 1$ is a polynomial.
    
    If additionally $1/h + 1/k = 1$, then the quadrilateral $Q$ degenerates to the common CGLS triangle $T_{1,h+k}$.
\end{thm}

\begin{thm}\label{thm:vertex}
	For an even integer $N \geq 4$, there exists a star-shaped polygon $P$ with $N$ vertices which are irrational points; that is, at least one coordinate is irrational.
\end{thm}

\begin{thm}\label{thm:edge}
	For $3 \leq N \neq 5$, there exists a star-shaped polygon $P$ with $N$ edges all of which have irrational slopes such that $P$ exhibits Ehrhart-theoretic period collapse.
\end{thm}

Our method is to start with an example of irrational period collapse in one dimension. Then we use the pyramid construction and a cut-and-paste procedure to produce irrational polygons with period collapse.

Potentially, this method can be generalized to higher dimensions. Moreover, because the CGLS triangles were discovered in the context of symplectic embeddings, it is conceivable that our new examples might also have applications in constructing extremal cases of symplectic embeddings. For the connection between Ehrhart theory and symplectic embeddings, see \cite{mcduff2012embedding}, \cite{cristofaro2013ehrhart}, \cite{cristofaro2018ghost} and \cite{schlenk2018symplectic}.

\subsection*{Acknowledgement}
The author benefits a lot from conversations with Dan Cristofaro-Gardiner, Sinai Robins and Richard Stanley. He thanks ICERM and its wonderful staff for their hospitality during part of this project. He would like to express his gratitude to Karim Adiprasito, whose grant \textit{ERC StG 716424 - CASe} supports this work.

\section{Period collapse in one dimension}
We explore the phenomenon of irrational period collapse in one dimension.

\begin{lem}
	Given a $1$-dimensional polytope $B = [a,b]$ with irrational numbers $a < b$ satisfying $b-a \in \mathbb{Z}$, we have $L_{B}(t)$ is a polynomial.
\end{lem}

\begin{proof}
	Suppose $b - a = n \in \mathbb{Z}_{>0}$. Since $t$ is an integer, both $ta$ and $tb$ are irrational. Therefore, $L_B(t) = \#([ta,tb] \cap \mathbb{Z}) = \lfloor tb \rfloor - \lfloor ta \rfloor = \lfloor ta + tn \rfloor - \lfloor ta \rfloor = \lfloor ta \rfloor + tn - \lfloor ta \rfloor = tn$ is a polynomial.
\end{proof}

The above $1$-dimensional polytope will be used as the base of the pyramid that will be constructed in the next section. Also, we believe Ehrhart-theoretic period collapse can be completely understood in one dimension, but it would require a bit more work.

\section{The pyramid construction}
We place $B$ on the horizontal line $x_2 = 1$ in $\mathbb{R}^2$. Then, we construct $P = \conv(B \cup \{(0,0)\})$ as the convex hull of $B$ and the origin. In other words, $P$ is the pyramid with base $B$ and apex the origin. Therefore, the Ehrhart function of the pyramid $P$ is related to that of the base $B$ in the following way:
\begin{equation}
	L_{P} (t) = 1 + \sum_{s=1}^t L_{B} (s) = 1 + \sum_{s=1}^t s n = \frac{n}{2} t^2 + \frac{n}{2} t + 1,
\end{equation} 
which is a polynomial! Therefore, we have already obtained a new example of irrational period collapse in two dimensions.

\section{A cut-and-paste procedure}
Let us describe a cut-and-paste procedure to produce more examples of irrational polygons with period collapse. First, we cut $P$ along an \textit{integral} edge $E$ to obtain two triangles $H_1$ and $H_2$. Then, we choose two integral affine transformations $\phi_1$ and $\phi_2$ such that the images $\phi_1 (H_1)$ and $\phi_2 (H_2)$ share a common edge, which is $\phi_1(E) = \phi_2(E)$, an integral segment. Then, we will see that the quadrilateral $Q = \phi_1 (H_1) \cup \phi_2 (H_2)$ is an irrational polygon whose Ehrhart function $L_Q (t)$ coincides with $L_{P(t)}$ and is a polynomial. 

To see this, we note that any integral affine transformation $\phi$ keeps the Ehrhart function invariant. By definition, $\phi$ is of the form $x \mapsto Ax + b$ with $A \in GL_2(\mathbb{Z})$ and $b \in \mathbb{Z}^2$. Because $\phi(\mathbb{Z}^2) = \mathbb{Z}^2$, we have $L_{\phi(P)} (t) = L_P (t)$ for any polytope $P \subset \mathbb{R}^2$. Therefore, we can compute
\begin{align}
	L_Q (t) &= L_{\phi_1(H_1)} (t) + L_{\phi_2(H_2)} (t) - L_{\phi_1(E)} (t) \nonumber \\
    	&= L_{H_1} (t) + L_{H_2} (t) - L_{E} (t) = L_{P} (t), 
\end{align} 
as stated in the previous paragraph.

Now we will show how to obtain the common CGLS triangles (as degenerate quadrilaterals) from the above procedure. 

\begin{proof}[Proof of Theorem \ref{thm:quad}]
Given two positive irrational numbers $h, k$ such that $h + k$ is an integer, we take $B = [-h,k]$ and construct the pyramid $P = \conv(\{(0,0)\} \cup B \times \{1\})$. Next, we cut $P$ along the edge $E = \conv\{(0,0), (0,1)\}$ into two triangles $H_1$ and $H_2$. We define
\[ \phi_1 (x) = \left(\begin{matrix} -1 & -1 \\ 0 & -1 \end{matrix}\right) x + \left(\begin{matrix} 1 \\ 1 \end{matrix}\right), \qquad
	\phi_2 (x) = \left(\begin{matrix} 0 & -1 \\ 1 & -1 \end{matrix}\right) x + \left(\begin{matrix} 1 \\ 1 \end{matrix}\right), \]
which satisfy $\phi_1(E) = \phi_2(E) = \conv\{(0,0),(1,1) \}$. Then, $$Q = Q_{h,k} =  \phi_1(H_1) \cup \phi_2(H_2) = \conv\{(0,0),(h,0),(1,1),(0,k)\}.$$ Observe that, if $h$ and $k$ satisfy the additional assumption $1/h + 1/k = 1$ as in the construction of common CGLS triangles, then the quadrilateral $Q$ degenerates to the triangle $T_{1,h+k}$.
\end{proof}

The exceptional CGLS triangles can also be obtained by a similar procedure; the difference is the apex of the pyramid $P$ is a rational point, instead of an integer point. The analysis will be a bit trickier, but not too hard.

We caution that the gluing process has to be taken with care. In fact, the union of two polygons (with disjoint interiors) with period collapse does not necessarily exhibit period collapse. For instance, consider a common CGLS triangle $T_{1,\beta}$ and denote $h,k$ to be the irrational solutions of the system $1/h + 1/k = 1, h + k = \beta$. Let $\rho_1$ be the reflection about the $x_1$-axis $x_2 = 0$ and $E = \conv \{(0,0), (h,0)\}$ the edge of $T_{1,\beta}$ on the $x_1$-axis. Then, we compute the Erhart function of the union $P = T_{1,\beta} \cup \rho_1(T_{1,\beta})$:
\begin{align}
	L_P(t) &= L_{T_{1,\beta}} (t) + L_{\rho_1 (T_{1,\beta})} (t) - L_E (t) = 2 L_{T_{1,\beta}} (t) - (\lfloor ht \rfloor + 1) \nonumber \\
    	&= (h+k) t^2 + (h+k) t - \lfloor ht \rfloor + 1.
\end{align}
This is not even a quasi-polynomial, and thus, the triangle $P$ does not exhibit period collapse.

\section{A more elaborate construction}
In order to prove Theorem \ref{thm:edge}, we employ a more elaborate construction than in the previous section. We divide the full angle at the origin into an even number of \textit{unimodular} sectors, and then, fill in each sector with an appropriate (possibly degenerate) quadrilateral $Q_{h,k}$ of Theorem \ref{thm:quad}.

We take positive irrational numbers $h_0,k_0$ such that $1/h_0 + 1/k_0 = 1$ and $h_0 + k_0$ is an integer. For example, we can choose $h_0 = \frac{5+\sqrt{5}}{2}$ and $k_0 = \frac{5-\sqrt{5}}{2}$ with $h_0 + k_0 = 5$ and $1/h_0 + 1/k_0 = 1$. We write $H^+ = \{h_0 + 1, h_0 + 2, \dots \}, K^+ = \{k_0 + 1, k_0 + 2, \dots \}$ and $H^\geq = \{h_0\} \cup H^+, K^\geq = \{k_0\} \cup K^+$. Given $h \in H^\geq$ and $k \in K^\geq$, we have
\begin{equation}
	h + k \in \mathbb{Z}_{>0}, \qquad 1/h + 1/k \leq 1/h + 1/k = 1,
\end{equation}
where, in the last inequality, the equality case happens if and only if $h = h_0$ and $k = k_0$. Thus, $Q_{h,k}$ is a quadrilateral, except that, when $(h,k) = (h_0,k_0)$, $Q_{h_0,k_0}$ is a triangle. 

Let us define a \textbf{unimodular sector} $\alpha(V_1, V_2)$ to be the angular region between two rays $\overrightarrow{OV_1}$ and $\overrightarrow{OV_2}$ which emanate from the origin $O$ and pass through two primitive integer points $V_1, V_2 \subset \mathbb{Z}^2$ which together form a unimodular matrix. Equivalently, we require the triangle $OV_1V_2$ to be an integer triangle with area $1/2$. Note that we can always subdivide a unimodular sector $\alpha(V_1, V_2)$ into two unimodular sectors $\alpha(V_1, V_1 + V_2)$ and $\alpha(V_1 + V_2, V_2)$. Therefore, we can subdivide the full angle at the origin into $k$ unimodular sectors for any $k \geq 3$. 

Now, let us divide the full angle at the origin into $2k$ unimodular sectors with $k \geq 2$. These sectors are separated by $2k$ rays $r_1, \cdots, r_{2k}$, which are indexed cyclically. Each ray $r_i$ is generated by a primitive integer vector $V_i$ and assigned a number $c_i$ such that, for odd $i$, $c_i \in H^\geq$, while for even $i$, $c_i \in K^\geq$. Therefore, our data is $\mathcal{D}_{2k} = \{(V_i, c_i): i = 1, \cdots, 2k\}$.

For each sector $\alpha_i$ whose extreme rays are $r_i$ and $r_{i+1}$, we fill it with the image of the quadrilateral $Q_{c_i,c_{i+1}}$ under the transformation by the unimodular matrix $M(V_i,V_{i+1})$, formed by the coordinates of $V_i$ and $V_{i+1}$. Notice that the edges of the quadrilaterals $M(V_i,V_{i+1}) Q_{c_i,c_{i+1}}$ on the rays $r_1, \cdots, r_{2k}$ match perfectly. Therefore, we can denote those edges $E_1, \cdots, E_{2k}$ and form the polygon $P$ as the union of all these quadrilaterals. 

\begin{lem}\label{lem:elab}
	Define the polytope $P = \bigcup_{i=1}^{2k} M(V_i,V_{i+1}) Q_{c_i,c_{i+1}}$. Then, $P$ is star-shaped, its edges have irrational slopes and its Ehrhart function $L_P(t)$ is a polynomial.
\end{lem}

\begin{proof}
Clearly, $P$ is star-shaped. Its edges contain two points $3$ of whose $4$ coordinates are integers, while the other is an irrational number. Therefore, all the edges of $P$ have irrational slopes.

We compute
\begin{align}
	L_P (t) &= \sum_{i=1}^{2k} L_{M(V_i,V_{i+1}) Q_{c_i,c_{i+1}}} (t) - \sum_{i=1}^{2k} L_{E_i} (t)  + L_{\{(0,0)\}} (t) \nonumber \\
    	&= \sum_{i=1}^{2k} L_{Q_{c_i,c_{i+1}}} (t) - \sum_{i=1}^{2k} L_{E_i} (t) + 1
\end{align}

Observe that $L_{E_i} (t) = \lfloor t(x_i+n_i) \rfloor + 1 = \lfloor tx_i \rfloor + tn_i + 1$ with $n_i \geq 0$ an integer and $x_i = h_0$ or $k_0$. Also, note that, if two irrational numbers $a,b > 0$ satisfy $a + b \in \mathbb{Z}$, then $\lfloor a \rfloor + \lfloor b \rfloor = a + b - 1$. Because half of the rays has $x_i = h_0$ and the other half has $x_i = k_0$, we have
\begin{align}
	\sum_{i=1}^{2k} L_{E_i} (t) &= 2k + t \sum_{i=1}^{2k} n_i + \sum_{i=1}^{k} (\lfloor th_0 \rfloor + \lfloor t k_0 \rfloor) \nonumber \\
    	&= 2k + t n_{total} + k (th_0 + tk_0 - 1) = (n_{total} + 5k) t + k
\end{align}

By Theorem \ref{thm:quad}, each $L_{Q_{c_i,c_{i+1}}} (t)$ is a polynomial in $t$. Therefore, we have proved that the Ehrhart function $L_P(t)$ is a polynomial, and thus, complete the proof of the lemma.
\end{proof}

To prove Theorems \ref{thm:vertex} and \ref{thm:edge}, the exact positions of the rays $r_i$ is not important, but we need to choose the numbers $c_i$ carefully so that $P$ has the right number of vertices/edges. Theorem \ref{thm:vertex} is easier than Theorem \ref{thm:edge}, so we start with its proof first.

\begin{proof}[Proof of Theorem \ref{thm:vertex}]
	Divide the full angle at the origin into $N = 2p$ unimodular sectors and assign to the rays $r_1, \dots, r_{N}$ the numbers $h_0$ and $k_0$, alternately. Then, we can produce from this data a polygon $P$ as the union of $N$ triangles. Then, $P$ has $N$ vertices which are irrational points, $P$ is star-shaped and its Ehrhart function is a polynomial by Lemma \ref{lem:elab}. This completes the proof.
\end{proof}

\begin{proof}[Proof of Theorem \ref{thm:edge}] The proof has $5$ steps.

	\textit{Step 0:} We show that if the theorem holds for $N$, then it holds for $N+4$
    
    Given the data $\mathcal{D}_{2k} = \{(V_1,c_1), \dots, (V_{2k},c_{2k})\}$ which yields a polygon $P$ with $N$ edges that satisfies Theorem \ref{thm:edge}, let us show that we can construct new data $\mathcal{D'}_{2k+2}$ that produces a new polygon $P'$ with $N+4$ edges and which also fulfills Theorem \ref{thm:edge}.
    
    Pick an index $i$, we will replace the pair $(V_i,c_i)$ in $\mathcal{D}_{2k}$ with three pairs $(V_i + V_{i-1}, c_i), (V_{i-1} + 2V_i + V_{i+1}, c'_{i}), (V_i + V_{i+1}, c_i)$ and obtain $\mathcal{D}'_{2k+2}$. Here, if $c_i \in H^\geq$, we take $c'_i$ to be any number in $K^+$; while if $c_i \in K^\geq$, then we take $c'_i \in K^+$ arbitrarily. In effect, we add two unimodular sectors which are filled by two quadrilaterals that contributes $4$ additional edges. Therefore, the polygon $P'$ constructed from the new data $\mathcal{D}'_{2k+2}$ has $N+4$ edges and satisfy Theorem \ref{thm:edge},by Lemma \ref{lem:elab}. This finishes Step 0.
    
    Thus, to complete the proof, we need to prove that the theorem is satisfied for $N = 4,6,7,9$.
    
    \textit{Step 1:} $N=4$.
    
    Take $P$ to be the union of the common CGLS triangle $Q_{h_0,k_0} = T_{1,5}$ and its reflected images about the two axes and about the origin. In other words, our data is $\mathcal{D}_4 = \{(e_1,h_0), (e_2,k_0), (-e_1,h_0), (-e_2,k_0)\}$, where $e_1,e_2$ are the standard basis vectors. Then, $P$ is a rhombus satisfying this theorem. By Step 0, this proves the theorem for $N = 4p \geq 4$.
    
    \textit{Step 2:} $N=6$.
    
    Take $\mathcal{D}_4 = \{(e_1,h_0), (e_2,k_0), (-e_1,h_0), (-e_2,k^+)\}$ with an arbitrary $k^+ \in K^+$. Since $Q_{h_0,k_0}$ is a triangle and $Q_{h^+,k^+}$ is a quadrilateral, the polygon $P$ constructed from $\mathcal{D}_4$ is the union of $2$ triangles and $2$ quadrilaterals and has $6$ edges. By Step 0, this proves the theorem for $N = 4p+2 \geq 6$.
    
    \textit{Step 3:} $N=7$.
    
    Take $\mathcal{D}_4 = \{(e_1,h_0), (e_2,k_0), (-e_1,h^+), (-e_2,k^+)\}$ with arbitrary $h^+ \in H^+$ and $k^+ \in K^+$. From $\mathcal{D}_4$, we produce a polygon $P$ which is the union of $1$ triangles and $3$ quadrilaterals and which has $7$ edges. By Step 0, this proves the theorem for $N = 4p+3 \geq 7$
    
    \textit{Step 4:} $N=9$. This case is a bit more complicated.
    
     Take $\mathcal{D}_6 = \{(e_1,h_0), (e_1+e_2,k_0), (e_2,h_0) (-e_1,k_0), (-e_1-e_2,h^+), (-e_2,k^+)\}$ with arbitrary $h^+ \in H^+$ and $k^+ \in K^+$. From $\mathcal{D}_6$, we produce a polygon $P$ which is the union of $3$ triangles and $3$ quadrilaterals and which has $9$ edges. By Step 0, this proves the theorem for $N = 4p+1 \geq 9$
     
     The proof is now complete.
    
\end{proof}

\section{Questions}

Here are further questions along the line of our results.

\begin{que}
	Regarding Theorem \ref{thm:vertex}, can we construct polygons with odd number of vertices which are all irrational points such that their Ehrhart functions are polynomials?
    
    In the proof of Theorem \ref{thm:vertex}, the ratio of the two coordinates of any vertex is always rational. Can we construct such polygons with (some or all) vertices whose coordinate ratios are irrational?
\end{que}

\begin{que}
	Regarding Theorem \ref{thm:edge}, can we construct triangles and pentagons whose edges have irrational slopes such that their Ehrhart functions are polynomials?    
\end{que}
    We suspect it is impossible to construct such triangles. The methods of \cite{} might be useful.

\begin{que}
	In the proofs of Theorems \ref{thm:vertex} and \ref{thm:edge}, some of the constructed polygons are convex, but there is no guarantee that all of them are. How can we consistently produce convex examples of irrational period collapse?
\end{que}


Of course, it will be very interesting to expand our constructions to higher dimensions.

\bibliographystyle{alpha}
\bibliography{sample}

\end{document}